\documentclass[final,11pt]{elsarticle}

\usepackage{verbatim,a4wide}

\usepackage{amsmath}
\usepackage{amssymb}
\usepackage{amsthm}

\usepackage{algorithm}
\usepackage{algpseudocode}

\usepackage{xcolor}

\usepackage{tikz}
\usetikzlibrary{patterns}

\usepackage{hyperref}

\usepackage[cp1250]{inputenc}
\usepackage[OT4]{fontenc}
\usepackage[english]{babel}

\usepackage{csquotes}

\numberwithin{equation}{section}
\numberwithin{algorithm}{section}
\numberwithin{figure}{section}
\numberwithin{table}{section}

\newtheorem{theorem}{Theorem}[section]
\newtheorem{lemma}[theorem]{Lemma}
\newtheorem{remark}[theorem]{Remark}

\newtheorem{example}[theorem]{Example}

\newtheorem{problem}[theorem]{Problem}

\algrenewcommand{\algorithmiccomment}[1]{\textbf{//}\,#1}

\begin{document}

\begin{frontmatter}

\title{Efficient evaluation of Bernstein-B\'{e}zier coefficients of
B-spline basis functions over one knot span}

\author[A1]{Filip Chudy\corref{cor}}
\ead{Filip.Chudy@cs.uni.wroc.pl}

\author[A1]{Pawe{\l} Wo\'{z}ny}
\ead{Pawel.Wozny@cs.uni.wroc.pl}

\cortext[cor]{Corresponding author.}

\address[A1]{Institute of Computer Science, University of Wroc{\l}aw,
             ul.~Joliot-Curie 15, 50-383 Wroc{\l}aw, Poland}


\begin{abstract}
New differential-recurrence relations for B-spline basis functions are given.
Using these relations, a recursive method for finding the Bernstein-B\'{e}zier
coefficients of B-spline basis functions over a single knot span is proposed.
The algorithm works for any knot sequence and has an asymptotically optimal 
computational complexity. Numerical experiments show that the new method gives 
results which preserve a high number of digits when compared to an approach 
which uses the well-known de Boor-Cox formula.

\end{abstract}

%
%
%
%


\end{frontmatter}

\section{Introduction}                                  \label{S:Introduction}

Recently, in~\cite{ChW2023} (see also \cite{FChPhD}), we have proposed a fast
recursive algorithm for computing the Bernstein-B\'{e}zier coefficients of all
B-spline (basis) functions of degree $m$ over all not-trivial knot spans
$[t_j,t_{j+1})\subset[t_0,t_n]$ $(m,n,j\in\mathbb N;\ 0\leq j<n)$ under the
natural assumption that no inner knot has multiplicity greater than $m$, i.e.,
the B-spline basis functions are continuous.
Our approach is very efficient due to a new differential-recurrence relation
for the B-spline functions of the same degree which implies the recurrence
relation for their Bernstein-B\'{e}zier coefficients.

If the boundary knot $t_n$ has multiplicity equal to $m+1$, the method is
asymptotically optimal. In that case, the computational complexity of the new
algorithm is $O(n_\varepsilon m^2)$, where $n_\varepsilon$ denotes the number
of not-empty knot spans. It means the number of floating point operations
is proportional to the number of coefficients which have to be found. If the
multiplicity of the knot $t_n$ is lower than $m+1$, the computational
complexity is increased by at most $O(m^3)$.

We have also shown that the algorithm is numerically efficient and can be used,
for example, in fast rendering of B-spline curves and surfaces. Note that the
new approach solves these problems with lower computational complexity than
using the famous de Boor-Cox algorithm. Furthermore, when the
Bernstein-B\'{e}zier form of a B-spline (basis) function is known, one can
evaluate it in linear time with respect to its degree by performing the
geometric algorithm proposed recently by the authors in \cite{WCh2020}.

Certainly, the Bernstein-B\'{e}zier form of all B-spline functions over all
not-trivial knot spans can also be found by other methods (cf., e.g.,
\cite{BC2021,BC2022,Boehm80,Boehm77,CR04,CLR80,RS04,Sablonniere78,TS2020}).
However, they have higher computational complexity as they rely on finding all
Bernstein-B\'{e}zier coefficients of all B-spline basis functions of degrees
lower than $m$. For details and the review of the possible solutions, see
\cite[pp.~3--4]{ChW2023}.

Let $m,n \in \mathbb{N}$. Suppose the knots $t_i$
$(-m\leq i\leq m+n)$ satisfy
\begin{equation}\label{E:KnotSequence}
 \underbrace{t_{-m} \leq \ldots \leq t_{-1} \leq t_0}_\text{boundary knots}
 \leq
 \underbrace{t_1 \leq \ldots \leq t_{n-1}}_\text{inner knots}
 \leq
 \underbrace{t_n \leq t_{n+1} \leq \ldots \leq t_{n+m}}_\text{boundary knots}
                                                              \qquad (t_0<t_n).
\end{equation}
The \textit{B-spline (basis) functions}
$N_{m,-m},N_{m,-m+1},\ldots,N_{m,n-1}$ of degree $m$ over $[t_0, t_n]$
are given by
\begin{equation}\label{E:DefBSpline}
N_{m,i}(u) := (t_{i+m+1}-t_i) [t_i, t_{i+1}, \ldots, t_{i+m+1}](t-u)^m_+,
\end{equation}
where the \textit{generalized divided difference}
(see, e.g., \cite[\S4.2]{DB} or \cite[p.~7]{Dierckx93}) acts on the
variable $t$, and
\begin{equation}\label{E:DefTruncatedPower}
(x-c)^m_+:=\left\{
               \begin{array}{ll}
                   (x-c)^m & (x \geq c),\\
                   0       & (x < c)
                \end{array}
            \right.
\end{equation}
is the \textit{truncated power function}.

Note that, for $i=0,1,\ldots,n-1$,
\begin{equation}\label{E:BSplineN0iExplicit}
N_{0,i}(u)=
\left\{
\begin{array}{ll}
 1 & (u \in [t_i, t_{i+1})),\\[1ex]
 0 & \mbox{otherwise},
\end{array}
\right.
\end{equation}
as well as
\begin{equation}\label{E:BSplineSupp}
\mbox{supp$(N_{m,i})$}=[t_i,t_{i+m+1}),
\end{equation}
where $-m\leq i<n$.

For more properties of the B-spline basis and for applications of splines in
computer-aided geometric design, approximation theory and numerical analysis,
see, e.g., \cite{DB, Dierckx93, Farin2002, Goldman, NURBS, Prautzsch}.

Let the \textit{adjusted Bernstein-B\'{e}zier basis form} of the B-spline
function $N_{m,i}$ $(-m\leq i<n)$ over a single non-empty \textit{knot span}
$[t_j, t_{j+1}) \subset [t_0, t_n]$ ($j=i,i+1,\ldots,i+m$) be
\begin{equation}\label{E:BSplineABB}
N_{m,i}(u)=\sum_{k=0}^m b^{(i,j)}_{m,k} B^m_k\left(\dfrac{u - t_j}
                                                     {t_{j+1} - t_j}\right)
                                                \qquad (t_j \leq u < t_{j+1}).
\end{equation}
Here $B^n_i$ $(0\leq i\leq n;\; n\in\mathbb N)$ is the \textit{$i$th Bernstein
(basis) polynomial of degree $n$} given by the formula
\begin{equation}\label{E:Def_BernPoly}
B^n_i(t):=\binom{n}{i} t^i (1-t)^{n-i}.
\end{equation}

In~\cite{ChW2023} (cf.~also \cite[\S3]{FChPhD}), the following problem has been
considered.

\begin{problem}\label{P:BSpline1}
Find the adjusted Bernstein-B\'{e}zier basis coefficients $b_{m,k}^{(i,j)}$
$(0\leq k\leq m)$ (cf.~\eqref{E:BSplineABB}) of all non-trivial functions
$N_{m,i}$ over all non-empty knot spans $[t_j,t_{j+1})\subset[t_0,t_n]$, i.e.,
for $j=0,1,\ldots,n-1$ and $i=j-m,j-m+1,\ldots,j$ (cf.~\eqref{E:BSplineSupp}).
\end{problem}

The proposed solution of Problem~\ref{P:BSpline1} uses the new
differential-recurrence relation for B-spline functions of the same degree.
Namely, in~\cite[Thm 3.1]{ChW2023}, it was proven that if the knot sequence is
\textit{clamped}, i.e.,
\begin{equation}\label{E:KnotClamped}
t_{-m} = t_{-m+1} = \ldots = t_0 < t_1 < \ldots < t_n =
                                                    t_{n+1} = \ldots = t_{n+m},
\end{equation}
then
\begin{equation}\label{E:BSplineDiffRecBase}
N_{m,i}(u) + \dfrac{t_i - u}{m} N_{m,i}'(u) =
      \dfrac{t_{m+i+1} - t_i}{t_{m+i+2} - t_{i+1}}
         \left(N_{m,i+1}(u) + \dfrac{t_{m+i+2} - u}{m} N_{m,i+1}'(u)\right),
\end{equation}
where $i=-m-1,-m+1,\ldots,n-1$, and a convention is adopted that
\begin{equation}\label{E:ConventionI}
N_{p,q}\equiv0\quad \mbox{for $q<-p$ or $q\geq n$}.
\end{equation}

Next, it has been observed that Eq.~\eqref{E:BSplineDiffRecBase} implies
the recurrence relation connecting the coefficients
$b_{m,k}^{(i,j)},b_{m,k+1}^{(i,j)}$ and
$b_{m,k}^{(i+1,j)},b_{m,k+1}^{(i+1,j)}$. See~\cite[Thm 4.3]{ChW2023}.

Using the mentioned facts, the fast and numerically efficient algorithm for
solving Problem~\ref{P:BSpline1} has been proposed.
See~\cite[Thm 4.4, Alg.~4.1 and \S6]{ChW2023}. Recall that if the boundary knot
$t_n$ has multiplicity equal to $m+1$ the method is asymptotically optimal ---
its computational complexity is $O(n_\varepsilon m^2)$, where $n_\varepsilon$
is the number of the non-trivial knot spans. Otherwise, the complexity
increases by at most $O(m^3)$.

\section{The problem}                                        \label{S:Problem}

At the end of our paper~\cite{ChW2023}, we wrote that
\enquote{Further research is required to find the application of the new 
differential-recurrence relation in finding the Bernstein-B\'{e}zier 
coefficients of B-spline basis functions over a single knot span}. The point is 
that the proposed approach is too expensive if we want to compute the 
Bernstein-B\'{e}zier form of B-spline functions over one fixed knot span only. 
In short, the main reason is as follows: to find the Bernstein-B\'{e}zier 
coefficients of B-spline functions over non-empty interval $[t_j,t_{j+1})$ 
using the recursive algorithm proposed by us, one has to previously compute the 
Bernstein-B\'{e}zier form of B-spline functions over all non-empty knot spans
$[t_{n-1},t_n),[t_{n-2},t_{n-1}),\ldots,[t_{j+1},t_{j+2})$.

Taking into account the mentioned drawback, the main goal of this article is
to propose an optimal method for solving the following problem.

\begin{problem}\label{P:BSpline2}
Suppose the knot span $[t_j,t_{j+1})$ is not empty with fixed $j$
$(0\leq j<n)$. Find the adjusted Bernstein-B\'{e}zier basis coefficients
$b_{m,k}^{(i,j)}$ $(0\leq k\leq m)$ (cf.~\eqref{E:BSplineABB}) of all
non-trivial functions $N_{m,i}$ over $[t_j,t_{j+1})$, i.e., for
$i=j-m,j-m+1,\ldots,j$ (cf.~\eqref{E:BSplineSupp}).
\end{problem}

Note that to solve Problem~\ref{P:BSpline2}, we have to compute $(m+1)^2$
Bernstein-B\'{e}zier basis coefficients $b_{m,k}^{(i,j)}$. In
Section~\ref{S:Solution}, using the results from \S\ref{S:NewResults},
we give the recursive algorithm for computing all these quantities in $O(m^2)$
time, which is asymptotically optimal from the computational complexity
standpoint. Certainly, the algorithm is useful, for example, when rendering
fragments of a B-spline curve or surface
(see also applications considered in~\cite[\S5]{ChW2023}).

Let us stress that one can solve Problem~\ref{P:BSpline2} using the methods
proposed, e.g., in
\cite{BC2021,BC2022,Boehm80,Boehm77,CR04,CLR80,RS04,Sablonniere78,TS2020}).
However they have higher computational complexity, i.e., at least $O(m^3)$.

\section{Some new results for B-spline basis functions}   \label{S:NewResults}

We start with proving two new properties of B-spline basis functions which
allow us to show that the differential-recurrence
relation~\eqref{E:BSplineDiffRecBase} holds even if the knot sequence is not
\textit{clamped} (cf.~\eqref{E:KnotClamped}).

\begin{lemma}\label{L:NewProperties}
For any knot sequence~\eqref{E:KnotSequence}, the B-spline functions $N_{m,i}$
satisfy the following identities:
\begin{eqnarray}
&&\label{E:NewPropertyI}
mN_{m,i}(u)+(t_{m+i+1}-u)N'_{m,i}(u)=
                      m(t_{m+i+1}-t_i)\frac{N_{m-1,i}(u)}
                                           {t_{m+i}-t_{i}},\\[1ex]
&&\label{E:NewPropertyII}
mN_{m,i}(u)+(t_{i}-u)N'_{m,i}(u)=m(t_{m+i+1}-t_i)\frac{N_{m-1,i+1}(u)}
                                                    {t_{m+i+1}-t_{i+1}},
\end{eqnarray}
where $-m\leq i\leq n-1$ (cf.~\eqref{E:ConventionI}) and we assume that for
any quantity $Q$
\begin{equation}\label{E:ConventionII}
\frac{Q}{t_k-t_\ell}:=0\quad \mbox{if the knots $t_k$ and $t_\ell$ are equal},
\end{equation}
as well as that $N_{-1,k}\equiv 0$ for all $k$.
\end{lemma}
\begin{proof}
It is easy to check that the lemma is true if $m\in\{0,1\}$. Let us fix $m>1$
and $i$ $(-m\leq i<n)$. First we prove~\eqref{E:NewPropertyII} assuming that
$t_{m+i+1}\neq t_{i+1}$. From the definitions~\eqref{E:DefBSpline}
and~\eqref{E:DefTruncatedPower}, we have
\begin{eqnarray*}
mN_{m,i}(u)&=&m(t_{m+i+1}-t_i)[t_i,t_{i+1},\ldots,t_{i+m+1}](t-u)_+^m\\
 &=&m(t_{m+i+1}-t_i)[t_i,t_{i+1}\ldots,t_{i+m+1}]
                                    \left\{(t-u)\cdot(t-u)_+^{m-1}\right\}.
\end{eqnarray*}
Using the Leibniz' formula for the generalized divided differences (see, e.g.,
\cite[Thm 4.4.12]{DB}), we obtain
\begin{eqnarray*}
mN_{m,i}(u)&=&m(t_{m+i+1}-t_i)
           \Big\{
                [t_i](t-u)\cdot[t_i,t_{i+1},\ldots,t_{i+m+1}](t-u)_+^{m-1}\\
                   &&\hspace{3cm}+
                       1\cdot[t_{i+1},t_{i+2},\ldots,t_{i+m+1}](t-u)_+^{m-1}+
                       0+\ldots+0
           \Big\}\\
 &=&-(t_i-u)N'_{mi}(u)+m(t_{m+i+1}-t_i)\frac{N_{m-1,i+1}(u)}
                                            {t_{i+m+1}-t_{i+1}}.
\end{eqnarray*}
Here the generalized divided difference acts on the variable $t$.

Now, if $t_{m+i+1}=t_{i+1}$ then
$[t_{i+1},\ldots,t_{i+m+1}](t-u)_+^{m-1}\equiv0$ (cf.~\eqref{E:KnotSequence})
and Eq.~\eqref{E:NewPropertyII} is also true (cf.~\eqref{E:ConventionII}).

The relation~\eqref{E:NewPropertyI} can be derived in a similar way.
\end{proof}

\begin{remark}
Observe that Eqs~\eqref{E:NewPropertyI} and~\eqref{E:NewPropertyII} can be seen
as {\rm fundamental properties} of B-spline basis functions. Indeed,
by adding~\eqref{E:NewPropertyI} multiplied by $u-t_i$
and~\eqref{E:NewPropertyII} multiplied by $t_{m+i+1}-u$ and next dividing
by $m(t_{m+i+1}-t_i)$, one can recover the famous de Boor-Cox recurrence
relation (see, e.g., \cite[Eqs.~(2.5)]{NURBS}), i.e.,
\begin{equation}\label{E:BSplineRecRel}
N_{m,i}(u)=(u-t_i)\frac{N_{m-1,i}(u)}{t_{m+i}-t_{i}}+
                 (t_{m+i+1}-u)\frac{N_{m-1,i+1}(u)}{t_{m+i+1}-t_{i+1}}
                                                      \qquad(-m\leq i<n).
\end{equation}

We can also easily obtain the following well-known representation of the
derivative of $N_{m,i}$ in terms of two B-spline basis functions of degree
$m-1$ (see, e.g., \cite[Eqs.~(2.7)]{NURBS}):
$$
N'_{m,i}(u)=m\left(\frac{N_{m-1,i}(u)}{t_{m+i}-t_{i}}-
                  \frac{N_{m-1,i+1}(u)}{t_{m+i+1}-t_{i+1}}\right)
                                                     \qquad(-m\leq i<n).
$$
It is enough to subtract~\eqref{E:NewPropertyII} from~\eqref{E:NewPropertyI}
and next divide both sides by $t_{m+i+1}-t_i$.

Note that in both cases, we use the conventions~\eqref{E:ConventionI}
and~\eqref{E:ConventionII}.
\end{remark}

Using Lemma~\ref{L:NewProperties}, one can observe that the assumption
of~\cite[Thm 3.1]{ChW2023} (cf.~\eqref{E:KnotClamped}) can be relaxed.
As a consequence, we obtain the following result.

\begin{theorem}\label{T:NewResult}
For any knot sequence~\eqref{E:KnotSequence}, where no inner knot has
multiplicity greater than $m$, the following differential-recurrence relation
for B-spline functions of the same degree $m$ holds true:
$$
N_{m,i}(u) + \dfrac{t_i - u}{m} N_{m,i}'(u) =
      \dfrac{t_{m+i+1} - t_i}{t_{m+i+2} - t_{i+1}}
         \left(N_{m,i+1}(u) + \dfrac{t_{m+i+2} - u}{m} N_{m,i+1}'(u)\right),
$$
where $i=-m-1,-m+1,\ldots,n-1$ (cf.~\eqref{E:ConventionI}).
\end{theorem}
\begin{proof}
It is enough to observe that the left-hand side of Eq.~\eqref{E:NewPropertyII}
divided by $m(t_{m+i+1}-t_i)$ is equal to the left-hand side
of~\eqref{E:NewPropertyI} for $i:=i+1$ divided by $m(t_{m+i+2}-t_{i+1})$.
\end{proof}

\begin{remark}\label{R:Assumption}
In the sequel, we assume that no inner knot $t_1, t_2, \ldots, t_{n-1}$ has
multiplicity greater than $m$. This natural assumption guarantees the B-spline
functions' continuity in $(t_0, t_n)$.

If an inner knot $t_i$ ($i=1, 2, \ldots, n-1$) has multiplicity greater than 
$m$, then every B-spline function of degree $m$ based on these knots
is identical to zero either for $t > t_i$ or for $t < t_i$.
Then the B-spline basis consists of two independent bases
and the problem can be divided into two smaller ones:
one for the non-empty knot spans in the interval $(t_0, t_i)$
and the other in $(t_i, t_n)$.

\end{remark}

Using Theorem~\ref{T:NewResult} and the approach presented in
\cite[\S4.2]{ChW2023}, one can construct the following recurrence relation for
the adjusted Bernstein-B\'{e}zier coefficients $b_{m,k}^{(i,j)}$
(see~\eqref{E:BSplineABB}) of the B-spline functions (cf.~also
\cite[Thm 4.3]{ChW2023}).

\begin{theorem}\label{T:MainRecurrence}
Suppose that the knot sequence has the form~\eqref{E:KnotSequence} and
the assumption in Remark~\ref{R:Assumption} holds true.
Let $0\leq j\leq n-1$ and $j-m<i<j$. Assume that the interval $[t_j,t_{j+1})$
is non-empty.

The adjusted Bernstein-B\'{e}zier coefficients $b_{m,k}^{(i,j)}$,
$$
N_{m,i}(u)=\sum_{k=0}^m b^{(i,j)}_{m,k}B^m_k\left(\dfrac{u-t_j}
                                                   {t_{j+1}-t_j}\right)
                                               \qquad (t_j \leq u < t_{j+1}),
$$
satisfy the following recurrence relation:
\begin{equation}\label{E:MainRecurrence}
b_{m,k}^{(i,j)}=\dfrac{t_j-t_i}{t_{j+1}-t_i}b_{m,k+1}^{(i,j)}+
                                    \dfrac{v_{mi}}{t_{j+1}-t_i}
                   \left(
                      (t_{j+1}-t_{m+i+2})b_{m,k}^{(i+1,j)}+
                      (t_{m+i+2}-t_j)b_{m,k+1}^{(i+1,j)}
                   \right),
\end{equation}
where $k=m-1,m-2,\ldots,0$, and $v_{mi}:=(t_{m+i+1}-t_i)/(t_{m+i+2}-t_{i+1})$.
\end{theorem}

\section{The solution}                                      \label{S:Solution}

Recall that we want to efficiently compute the coefficients $b_{m,k}^{(i,j)}$
for all $k=0,1,\ldots,m$ and all $i=j-m,j-m+1,\ldots,j$, where
$0\leq j\leq n-1$ is fixed and $t_j\neq t_{j+1}$ (see Problem~\ref{P:BSpline2})
assuming that the knot sequence has the form~\eqref{E:KnotSequence} and no
inner knot has multiplicity greater than $m$ (cf.~Remark~\ref{R:Assumption}).

If $m=0$ the problem is trivial, because $b^{(i,j)}_{0,0}=\delta_{i,j}$
(cf.~\eqref{E:BSplineN0iExplicit}). Here $\delta_{i,j}$ denotes the
\textit{Kronecker delta}, i.e.,
\begin{equation}\label{E:KroneckerDelta}
\mbox{$\delta_{i,i}:=1$ and $\delta_{i,j}:=0$ for $i\neq j$}.
\end{equation}

Certainly, in general, if $i<j-m$ or $i>j$, we have
$$
b^{(i,j)}_{m,k}=0 \qquad (0\leq k\leq m)
$$
(cf.~\eqref{E:BSplineSupp}).

\subsection{The main idea}                                     \label{SS:Idea}

Assume that $m>0$. Let us fix $0\leq j\leq n-1$ $(t_j\neq t_{j+1})$ and
$j-m<i<j$. Supposing that the coefficients $b_{m,k}^{(i+1,j)}$ for all
$0\leq k\leq m$ have just been computed and that the value of $b_{m,m}^{(i,j)}$
is known, one can use the recurrence relation~\eqref{E:MainRecurrence} to
compute the coefficients
$$
b_{m,m-1}^{(i,j)}, b_{m,m-2}^{(i,j)}, \ldots, b_{m0}^{(i,j)}
$$
in $O(m)$ time.

It means that to solve Problem~\ref{P:BSpline2} using
Theorem~\ref{T:MainRecurrence} in $O(m^2)$ time --- i.e., to propose the
asymptotically optimal algorithm of computing the adjusted Bernstein-B\'{e}zier
forms of all non-trivial $N_{m,i}$ over non-empty knot span $[t_j,t_{j+1})$ for
all $j-m\leq i\leq j$ --- first, one needs to:
\begin{enumerate}
\itemsep 1ex

\item find the adjusted Bernstein-B\'{e}zier forms of $N_{m,j-m}$ and $N_{m,j}$,

\item evaluate the quantities $b_{m,m}^{(i,j)}$ for all
      $i=j-m+1,j-m+2,\ldots,j-1$

\end{enumerate}
in at most $O(m^2)$ time. We will discuss these issues in the next subsection.

\subsection{The initial conditions}                            \label{SS:Init}

Using~\eqref{E:BSplineRecRel} and~\eqref{E:BSplineSupp}, it is not difficult to
check that for $m>0$ and any knot sequence~\eqref{E:KnotSequence}, we have
\begin{eqnarray*}
&&N_{m,j-m}(u)=\dfrac{(t_{j+1} - t_j)^{m-1}}
                   {\prod_{k=2}^{m}(t_{j+1}-t_{j+1-k})}
                        B^m_0\left(\dfrac{u-t_j}{t_{j+1}-t_j}\right),\\[1.5ex]
&&N_{m,j}(u)=\dfrac{(t_{j+1}-t_j)^{m-1}}
                  {\prod_{k=2}^m(t_{j+k}-t_j)}
                        B^m_m\left(\dfrac{u-t_j}{t_{j+1}-t_j}\right)
                                       \qquad\qquad(t_j\leq u <t_{j+1}),
\end{eqnarray*}
where $0\leq j\leq n-1$ and $t_j\neq t_{j+1}$.

It means that the Bernstein-B\'{e}zier coefficients of $N_{m,j-m}$ and $N_{m,j}$
$(m>0)$ in the non-empty knot span $[t_j,t_{j+1})$ $(0\leq j\leq n-1)$ are as
follows:
\begin{eqnarray}
&&\label{E:BSplineFirstLastExp}
  b_{m,0}^{(j-m,j)}=\dfrac{(t_{j+1}-t_j)^{m-1}}
                       {\prod_{k=2}^{m}(t_{j+1}-t_{j-k+1})},\qquad
  b_{m,k}^{(j-m,j)}=0\quad (k=1,2,\ldots,m),\\
&&\label{E:BSplineFirstLastExp-II}
  b_{m,k}^{(j,j)}=0\quad (k=0,1,\ldots,m-1),\qquad
  b_{m,m}^{(j,j)}=\dfrac{(t_{j+1}-t_j)^{m-1}}{\prod_{k=2}^m (t_{j+k}-t_j)}
\end{eqnarray}
(cf.~\cite[\S4.1]{ChW2023}) and all these coefficients can be computed in
total $O(m)$ time.

Now, we show that the quantities $b^{(i,j)}_{m,m}$ satisfy a simple recurrence
relation with respect to $m$ and $i$.

\begin{theorem}\label{T:InitialRecurrence}
For $m>0$ and any knot sequence~\eqref{E:KnotSequence}, where no inner knot has
multiplicity greater than $m$, the following recurrence relation holds true:
\begin{equation}\label{E:InitialRecurrence}
b_{m,m}^{(i,j)}=\dfrac{t_{j+1}-t_i}
                      {t_{m+i}-t_i}b_{m-1,m-1}^{(i,j)}+
                \dfrac{t_{m+i+1}-t_{j+1}}
                      {t_{m+i+1}-t_{i+1}}b_{m-1,m-1}^{(i+1,j)},
\end{equation}
where $0\leq j\leq n-1$ is fixed, $t_j\neq t_{j+1}$, and
$i=j-m,j-m+1,\ldots,j-1$.
\end{theorem}
\begin{proof}
In the proof we use the approach similar to the one presented
in~\cite{CR04,RS04}.

It is well-known that Bernstein polynomials~\eqref{E:Def_BernPoly} satisfy
\begin{eqnarray}
\label{E:BernT}
tB^n_k(t)&=&\dfrac{k+1}{n+1} B^{n+1}_{k+1}(t),\\
\label{E:BernDegUp}
B^n_k(t)&=&\dfrac{n-k+1}{n+1} B^{n+1}_{k}(t)
            +\dfrac{k+1}{n+1} B^{n+1}_{k+1}(t),
\end{eqnarray}
where $0\leq k\leq n$.

Let $u\in[t_j,t_{j+1})$, where $0\leq j\leq n-1$ and $t_j\neq t_{j+1}$. Set
$$
t\equiv t^{(j)}(u):=\dfrac{u-t_j}{t_{j+1}-t_j}
$$
(cf.~\eqref{E:BSplineABB}). Certainly, $t\in[0,1)$, and
\begin{equation}\label{E:UofT}
u=(t_{j+1}-t_j)t+t_j.
\end{equation}

Assume $i=j-m,j-m+1,\ldots,j-1$ (cf.~\eqref{E:BSplineSupp}).
Applying~\eqref{E:BSplineABB} to~\eqref{E:BSplineRecRel} gives
$$
\sum_{k=0}^m b^{(i,j)}_{m,k} B^m_k(t)
 =\dfrac{u-t_i}{t_{m+i}-t_i}\sum_{k=0}^{m-1}b^{(i,j)}_{m-1,k}B^{m-1}_k(t)
        +\dfrac{t_{m+i+1}-u}{t_{m+i+1}-t_{i+1}}
                     \sum_{k=0}^{m-1}b^{(i+1,j)}_{m-1,k}B^{m-1}_k(t).
$$
Using~\eqref{E:UofT} gives, after some algebra,
\begin{multline*}
\sum_{k=0}^m b^{(i,j)}_{m,k}B^m_k(t)=
  \sum_{k=0}^{m-1}\left(\dfrac{b^{(i,j)}_{m-1,k}}
                             {t_{m+i}-t_i}
  - \dfrac{b^{(i+1,j)}_{m-1,k}}
          {t_{m+i+1}-t_{i+1}}\right)(t_{j+1}-t_j)tB^{m-1}_k(t)\\
  + \sum_{k=0}^{m-1}\left(\dfrac{b^{(i,j)}_{m-1,k}(t_j - t_i)}{t_{m+i}-t_i}
  + \dfrac{b^{(i+1,j)}_{m-1,k}(t_{m+i+1}-t_j)}
          {t_{m+i+1}-t_{i+1}}\right)B^{m-1}_k(t).
\end{multline*}
Applying~\eqref{E:BernDegUp} and~\eqref{E:BernT} gives, then, after some more
algebra,
\begin{multline}\label{E:DeBoorCoxMethod}
\sum_{k=0}^m b^{(i,j)}_{m,k}B^m_k(t) =
  \sum_{k=0}^{m-1}\left(\dfrac{b^{(i,j)}_{m-1,k}(t_j-t_i)}{t_{m+i}-t_i}
  + \dfrac{ b^{(i+1,j)}_{m-1,k}(t_{m+i+1}-t_j)}{t_{m+i+1}-t_{i+1}}
                  \right)
                                                  \dfrac{m-k}{m}B^{m}_{k}(t)\\
  + \sum_{k=1}^{m}\left(
  \dfrac{b^{(i,j)}_{m-1,k-1}(t_{j+1}-t_i)}{t_{m+i}-t_i}
  + \dfrac{b^{(i+1,j)}_{m-1,k-1}(t_{m+i+1}-t_{j+1})}{t_{m+i+1}-t_{i+1}}
                  \right)\dfrac{k}{m} B^{m}_{k}(t).
\end{multline}
Matching the coefficients for $B^m_m(t)$ justifies the recurrence
relation~\eqref{E:InitialRecurrence}.
\end{proof}

Observe that it is possible to evaluate the quantities $b^{(i,j)}_{m,m}$ for
$m>1$ and all $i=j-m+1,j-m+2,\ldots,j-1$ in $O(m^2)$ time using the recurrence
relation~\eqref{E:InitialRecurrence}, with the initial conditions
\begin{eqnarray}
\label{E:InitCondInitialRecurrenceI}
&&b_{p,p}^{(i,j)}:=0 \quad (i<j-p),\qquad
                                       b_{0,0}^{(i,j)}:=\delta_{i,j},\\[1.25ex]
\label{E:InitCondInitialRecurrenceII}
&&
b_{1,1}^{(j,j)}:=1,\qquad
b_{p,p}^{(j,j)}:=\dfrac{t_{j+1}-t_j}
                       {t_{j+p}-t_j}\cdot b_{p-1,p-1}^{(j,j)} \quad (p>0)
\end{eqnarray}
(see~\eqref{E:BSplineSupp}, \eqref{E:KroneckerDelta},
and~\eqref{E:BSplineFirstLastExp-II}).

\begin{remark}\label{R:DeBoorCoxMethod}
From Eq.~\eqref{E:DeBoorCoxMethod}, it follows that
\begin{eqnarray*}
b^{(i,j)}_{m,k}&=&\dfrac{m-k}{m}
     \left(\dfrac{b^{(i,j)}_{m-1,k}(t_j-t_i)}{t_{m+i}-t_i}
                +\dfrac{ b^{(i+1,j)}_{m-1,k}(t_{m+i+1}-t_j)}
                                            {t_{m+i+1}-t_{i+1}}\right)\\
  &&\hspace{3cm}
   +\dfrac{k}{m}
       \left(\dfrac{b^{(i,j)}_{m-1,k-1}(t_{j+1}-t_i)}{t_{m+i}-t_i}
                     +\dfrac{b^{(i+1,j)}_{m-1,k-1}(t_{m+i+1}-t_{j+1})}
                                                  {t_{m+i+1}-t_{i+1}}\right),
\end{eqnarray*}
where $0\leq k\leq m$ (cf.~the convention~\eqref{E:ConventionII}). Observe that
one can use this recurrence scheme to solve Problem~\ref{P:BSpline2} in $O(m^3)$
time.
\end{remark}

\subsection{The algorithm}                                \label{SS:Algorithm}

The results presented in Theorem~\ref{T:MainRecurrence}, as well as
in~\S\ref{SS:Idea} and~\S\ref{SS:Init} can be combined to prove the following
theorem.

\begin{theorem}\label{T:BSplineBezierBig}
Suppose that $m>0$, the knot sequence has the form~\eqref{E:KnotSequence} and
no inner knot has multiplicity greater than $m$. Let us fix $j$ $(0\leq j<n)$.
Assume that knot span $[t_j, t_{j+1})$ is non-empty.

The $(m+1)^2$ adjusted Bernstein-B\'{e}zier coefficients $b_{m,k}^{(i,j)}$
(cf.~\eqref{E:BSplineABB}) of the basis B-spline functions
$N_{m,i}$~\eqref{E:DefBSpline} over the knot span $[t_j, t_{j+1})$, for
$i=j-m,j-m+1,\ldots,j$ and $k=0,1,\ldots,m$, can be computed in the
computational complexity $O(m^2)$ in the following~way:
\begin{enumerate}
\itemsep1ex

\item Set $b_{0,0}^{(j,j)}=1$. Then, for $p=1,2,\ldots,m$, compute
      $b_{p,p}^{(j,j)}=\dfrac{t_{j+1}-t_j}{t_{j+p}-t_j}b_{p-1,p-1}^{(j,j)}$.

\item For $p=1,2,\ldots,m$ and $i=j-1,j-2,\ldots,j-p$ compute $b_{p,p}^{(i,j)}$
      using Eqs~\eqref{E:InitialRecurrence},
      \eqref{E:InitCondInitialRecurrenceI} and the
      convention~\eqref{E:ConventionII}.

\item For $k=0,1,2,\ldots,m$, compute $b_{m,k}^{(j-m,j)}$
      using Eq.~\eqref{E:BSplineFirstLastExp}.

\item For $k=m-1, m-2, \ldots, 0$ and $i=j-1, j-2, \ldots, j-m+1$,
      compute the coefficients $b_{m,k}^{(i,j)}$ using
      Eq.~\eqref{E:MainRecurrence}.

\end{enumerate}
\end{theorem}

\begin{figure*}[ht!]
  \centering
  \begin{tikzpicture}
  \foreach \row in {0,...,6}
    \draw (\row, 0) -- (\row, 6);
  \foreach \col in {0,...,6}
    \draw (0, \col) -- (6, \col);

  \foreach \y in {4,...,4}
  {
    \draw[->] (0.8, \y + 0.5) -- (1.2, \y + 0.5);
    \draw[->] (0.8, \y + 1.2) -- (1.2, \y + 0.8);
  }
  \foreach \y in {3,...,4}
  {
    \draw[->] (1.8, \y + 0.5) -- (2.2, \y + 0.5);
    \draw[->] (1.8, \y + 1.2) -- (2.2, \y + 0.8);
  }
  \foreach \y in {2,...,4}
  {
    \draw[->] (2.8, \y + 0.5) -- (3.2, \y + 0.5);
    \draw[->] (2.8, \y + 1.2) -- (3.2, \y + 0.8);
  }
  \foreach \y in {1,...,4}
  {
    \draw[->] (3.8, \y + 0.5) -- (4.2, \y + 0.5);
    \draw[->] (3.8, \y + 1.2) -- (4.2, \y + 0.8);
  }
  \foreach \y in {0,...,4}
  {
    \draw[->] (4.8, \y + 0.5) -- (5.2, \y + 0.5);
    \draw[->] (4.8, \y + 1.2) -- (5.2, \y + 0.8);
  }
  \foreach \x in {0,...,4}
  {
  \draw[-] (\x + 0.5, 5.8) -- (\x + 0.5, 6.2) -- (\x + 1.2, 6.2);
  \draw[->] (\x + 1.2, 6.2) -- (\x + 1.2, 5.8);
  }
  \node[above] at (3, 6.2) {\mbox{Eq.~\eqref{E:InitCondInitialRecurrenceII}}};

  \node[left] at (0, 0.5) {$i=j-5$};
  \node[left] at (0, 1.5) {$i=j-4$};
  \node[left] at (0, 2.5) {$i=j-3$};
  \node[left] at (0, 3.5) {$i=j-2$};
  \node[left] at (0, 4.5) {$i=j-1$};
  \node[left] at (0, 5.5) {$i=j$};

  \node[below] at (0.5, 0) {$p=0$};
  \node[below] at (1.5, 0) {$1$};
  \node[below] at (2.5, 0) {$2$};
  \node[below] at (3.5, 0) {$3$};
  \node[below] at (4.5, 0) {$4$};
  \node[below] at (5.5, 0) {$5$};

  \foreach \y in {0,...,4}
      \node at (0.5, \y + 0.5) {$0$};
  \foreach \y in {0,...,3}
      \node at (1.5, \y + 0.5) {$0$};
  \foreach \y in {0,...,2}
      \node at (2.5, \y + 0.5) {$0$};
  \foreach \y in {0,...,1}
      \node at (3.5, \y + 0.5) {$0$};
  \node at (4.5, 0.5) {$0$};

  \node at (0.5, 5.5) {$1$};
\end{tikzpicture}
\caption{An illustration of using Eq.~\eqref{E:InitialRecurrence}
        (cf.~\eqref{E:InitCondInitialRecurrenceI}) for $m=5$.}\label{F:Initial}
\end{figure*}
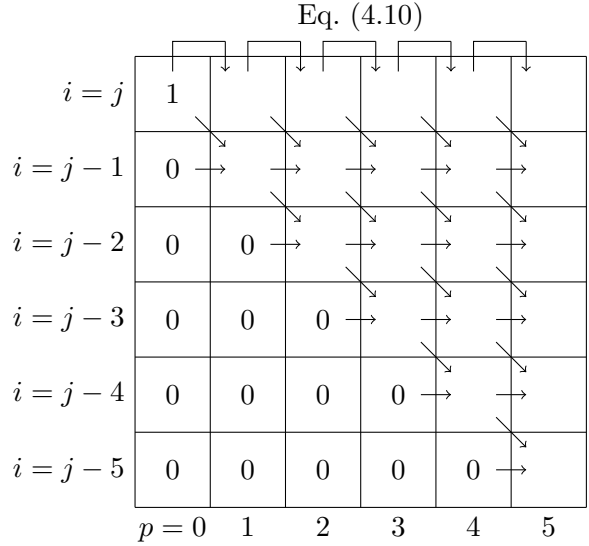

Figure~\ref{F:Initial} illustrates the computations done in the first two steps
of Theorem~\ref{T:BSplineBezierBig}. The rightmost column is then inserted into
the rightmost column of Figure~\ref{F:Full}, which illustrates the remaining
steps.

\begin{figure*}[ht!]
  \centering
  \begin{tikzpicture}
  \foreach \row in {0,...,6}
    \draw (\row, 0) -- (\row, 6);
  \foreach \col in {0,...,6}
    \draw (0, \col) -- (6, \col);

  \foreach \x in {0,...,4}
  {
    \foreach \y in {1,...,4}
    {
      \draw[->] (\x + 1.2, \y + 0.5) -- (\x + 0.8, \y + 0.5);
      \draw[->] (\x + 1.2, \y + 1.2) -- (\x + 0.8, \y + 0.8);
      \draw[->] (\x + 0.5, \y + 1.2) -- (\x + 0.5, \y + 0.8);
    }
  }

  \node[right] at (7.5, 3.0) {Last column of Fig.~\ref{F:Initial}};
  \draw[-] (6.5, 3.0) -- (7.5, 3.0);
  \draw[-] (6.5, 5.5) -- (6.5, 1.5);
  \foreach \y in {1,...,5}
    \draw[->] (6.5, \y + 0.5) -- (5.5, \y + 0.5);

  \node[right] at (7.5, 0.25) {Eq.~\eqref{E:BSplineFirstLastExp}};
  \draw[-] (7.5, 0.25) -- (0.5, 0.25);
  \draw[->] (0.5, 0.25) -- (0.5, 0.5);

  \foreach \x in {0,...,4}
  {
    \node at (\x + 0.5, 5.5) {$0$};
  }
  \foreach \x in {1,...,5}
  {
    \node at (\x + 0.5, 0.5) {$0$};
  }

  \node[left] at (0, 0.5) {$i=j-5$};
  \node[left] at (0, 1.5) {$i=j-4$};
  \node[left] at (0, 2.5) {$i=j-3$};
  \node[left] at (0, 3.5) {$i=j-2$};
  \node[left] at (0, 4.5) {$i=j-1$};
  \node[left] at (0, 5.5) {$i=j$};

  \node[below] at (0.5, 0) {$k=0$};
  \node[below] at (1.5, 0) {$1$};
  \node[below] at (2.5, 0) {$2$};
  \node[below] at (3.5, 0) {$3$};
  \node[below] at (4.5, 0) {$4$};
  \node[below] at (5.5, 0) {$5$};
\end{tikzpicture}
\caption{An illustration of the full recurrence scheme using
         Theorem~\ref{T:MainRecurrence} for $m=5$.}\label{F:Full}
\end{figure*}
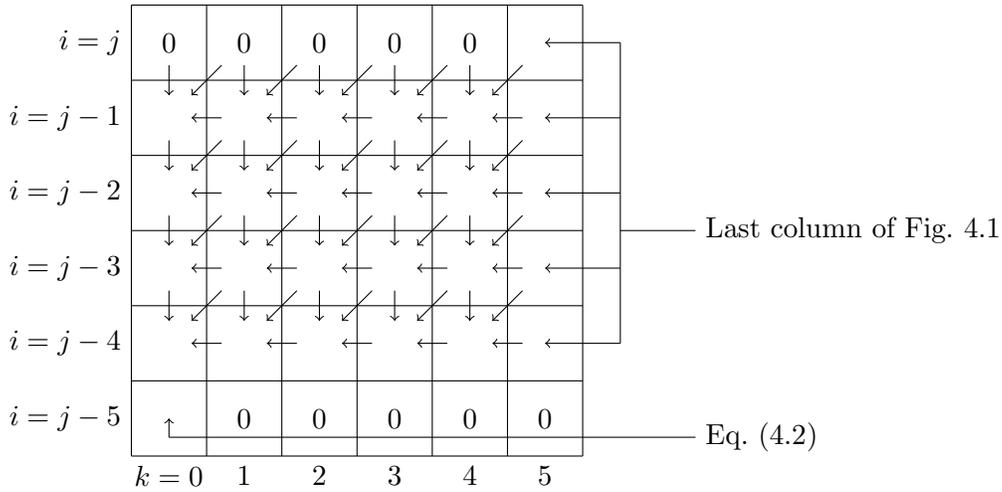

Algorithm~\ref{A:BSplineBezForm} implements the approach proposed in
Theorem~\ref{T:BSplineBezierBig}. This algorithm returns an array
$B\equiv B[0..m, j-m..j]$, where
$$
B[k,i]=b^{(i,j)}_{m,k}\qquad (j-m\leq i<j,\; 0\leq k\leq m)
$$
(cf.~\eqref{E:BSplineABB}).

The complexity of Algorithm~\ref{A:BSplineBezForm} is $O(m^2)$ --- giving the
optimal $O(1)$ time per coefficient.

To examine the numerical behavior of the new algorithm (complexity $O(m^2)$),
we have compared it to the method mentioned in Remark~\ref{R:DeBoorCoxMethod}
(complexity $O(m^3)$). Numerical experiments show that both methods are stable
and give almost the same results.
As expected, the new algorithm is usually faster, especially for larger $m$.
The source code in C which was used to perform the tests is available at
\url{https://github.com/filipchudy/bspline-coeffs-one-interval}.
The results have been obtained on a computer with \texttt{Intel
Core i5-6300U CPU at 2.40GHz} processor and 8GB RAM, using
\texttt{GNU C Compiler 13.2.0} (double precision).

\begin{example}\label{Ex:Numerical}
An experiment which checks the numerical quality of the method given 
in~\cite{ChW2023} and the new algorithm has been done. For 
$m \in \{3, 4, 5, 10, 20, 30, 50\}$ and $n \in \{10, 50, 100\}$, the following 
procedure has been repeated 50,000 times.

A sequence of knots has been generated in the following way. A sequence of 
knots has been randomized using the \texttt{rand} C function. The first knot is 
sampled from the $[-10, 10]$ interval. Then, the following knots are sampled, 
with randomized multiplicities and distances, from $1, 2, \ldots, m$ and 
$(0, 0.5)$, respectively. Afterwards, all right-hand side boundary knots are 
clamped.

Then, all Bernstein-B\'{e}zier coefficients of the B-spline functions are 
evaluated using the method given in~\cite{ChW2023}, the new algorithm, and the 
approach given in Remark~\ref{R:DeBoorCoxMethod}. The evaluation has been done 
for each knot span. Table~\ref{T:Numerical} shows the mean number of correct 
decimal digits for the first two methods, with reference results being the ones 
obtained by the third method, which uses the de Boor-Cox formula. The mean 
number is over all $(m + 1)^2$ Bernstein-B\'{e}zier coefficients of B-spline 
functions in each knot span, over all non-empty knot spans, for 50,000 
different knot sequences.
\end{example}

\begin{table*}[ht!]\small
\begin{center}
\renewcommand{\arraystretch}{1.3}
\begin{tabular}{lccc}
$m$ & $n$ & mean~\cite{ChW2023} & mean new \\ \hline
3 & 10 &  13.930 &  17.858 \\
3 & 50 &  13.952 &  17.859 \\
3 & 100 &  13.959 &  17.860 \\ \hline
4 & 10 &  13.895 &  17.762 \\
4 & 50 &  13.902 &  17.763 \\
4 & 100 &  13.906 &  17.764 \\ \hline
5 & 10 &  13.817 &  17.650 \\
5 & 50 &  13.805 &  17.653 \\
5 & 100 &  13.804 &  17.654 \\ \hline
10 & 10 &  13.533 &  17.310 \\
10 & 50 &  13.402 &  17.317 \\
10 & 100 &  13.382 &  17.319 \\ \hline
20 & 10 &  13.642 &  17.031 \\
20 & 50 &  13.072 &  16.943 \\
20 & 100 &  12.985 &  16.946 \\ \hline
30 & 10 &  13.792 &  16.905 \\
30 & 50 &  12.861 &  16.646 \\
30 & 100 &  12.706 &  16.651 \\ \hline
50 & 10 &  13.903 &  16.775 \\
50 & 50 &  12.532 &  16.137 \\
50 & 100 &  12.303 &  16.150 \\ \hline
\end{tabular}
\renewcommand{\arraystretch}{1}
\caption{The mean number of correct digits for Example~\ref{Ex:Numerical}.
The source code in C which was used to perform
the tests is available at
\url{https://github.com/filipchudy/bspline-coeffs-one-interval}.}\label{T:Numerical}
\vspace{-3ex}
\end{center}
\end{table*}

\begin{example}\label{Ex:Time}
A comparison of running times for the new algorithm and the method given in 
Remark~\ref{R:DeBoorCoxMethod} has been done. For 
$m \in \{3, 4, 5, 10, 20, 30, 50\}$ and $n \in \{10, 50, 100\}$, the following 
procedure has been repeated 50,000 times.

A sequence of knots has been generated in the following way. A sequence of 
knots has been randomized using the \texttt{rand} C function. The first knot is 
sampled from the $[-10, 10]$ interval. Then, the following knots are sampled, 
with randomized multiplicities and distances, from $1, 2, \ldots, m$ and 
$(0, 0.5)$, respectively. Unlike Example~\ref{Ex:Numerical}, the right-hand 
side boundary knots are not clamped.

Then, all Bernstein-B\'{e}zier coefficients of the B-spline functions over all 
non-empty knot spans are evaluated using the new algorithm and the method given 
in Remark~\ref{R:DeBoorCoxMethod}. Table~\ref{T:Time} shows the comparison of 
running times in seconds, for different $m, n$, for all 50,000 tests.
\end{example}

\begin{table*}[ht!]\small
\begin{center}
\renewcommand{\arraystretch}{1.3}
\begin{tabular}{lccc}
$m$ & $n$ & new & deBoor \\ \hline
3 & 10 & 0.158 & 0.180 \\
3 & 50 & 0.523 & 0.643 \\
3 & 100 & 0.962 & 1.206 \\ \hline
4 & 10 & 0.190 & 0.238 \\
4 & 50 & 0.668 & 0.921 \\
4 & 100 & 1.261 & 1.770 \\ \hline
5 & 10 & 0.226 & 0.309 \\
5 & 50 & 0.843 & 1.270 \\
5 & 100 & 1.602 & 2.448 \\ \hline
10 & 10 & 0.404 & 0.847 \\
10 & 50 & 1.730 & 3.916 \\
10 & 100 & 3.377 & 7.742 \\ \hline
20 & 10 & 0.746 & 2.670 \\
20 & 50 & 3.561 & 13.452 \\
20 & 100 & 7.152 & 27.334 \\ \hline
30 & 10 & 1.099 & 5.635 \\
30 & 50 & 5.397 & 28.847 \\
30 & 100 & 10.766 & 57.869 \\ \hline
50 & 10 & 1.811 & 14.814 \\
50 & 50 & 9.176 & 77.182 \\
50 & 100 & 18.226 & 153.975 \\ \hline
\end{tabular}
\renewcommand{\arraystretch}{1}
\caption{Running times comparison (in seconds) for
Example~\ref{Ex:Time}. The source code in C which was used to perform
the tests is available at
\url{https://github.com/filipchudy/bspline-coeffs-one-interval}.}\label{T:Time}
\vspace{-3ex}
\end{center}
\end{table*}

\begin{algorithm}[ht!]
\caption{Computing the coefficients of the adjusted Bernstein-B\'{e}zier form
         of the B-spline functions over one knot span}\label{A:BSplineBezForm}
\begin{algorithmic}[1]
\Procedure {BSplineBBFOneKnotSpan}{$n, m, j, [t_{-m},t_{-m+1},\ldots,t_{n+m}]$}
\State $Bm \gets \texttt{Array[0..m, j-m..j](fill=0)}$
\State $Bm[0, j] \gets 1$
\For {$p=1,m$}
  \State $Bm[p, j] \gets \dfrac{t_{j+1}-t_j}{t_{j+p}-t_j}\times Bm[p-1, j]$
\EndFor

\For {$p=1,m$}
  \For {$i=j-1,j-p$}
      \State \algorithmiccomment{With the convention~\eqref{E:ConventionII}.}
      \State $Bm[p,i] \gets \dfrac{t_{j+1}-t_i}{t_{p+i}-t_i} \times Bm[p-1, i]
                          +\dfrac{t_{p+i+1}-t_{j+1}}{t_{p+i+1}-t_{i+1}}
                                                           \times Bm[p-1, i+1]$
  \EndFor
\EndFor
\State $B \gets \texttt{Array[0..m, j-m..j](fill=0)}$
\State $B[0, j-m] \gets \dfrac{(t_{j+1}-t_j)^{m-1}}
                              {\prod_{k=2}^m (t_{j+1}-t_{j-k+1})}$
\State $B[m, j-m..j] \gets Bm[m, j-m..j]$
\For {$k=m-1,0$}
  \For {$i=j-1, j-m+1$}
    \State $v_{mi} \gets \dfrac{t_{m+i+1}-t_i}{t_{m+i+2}-t_{i+1}}$
    \State $c_i \gets (t_{j+1}-t_{m+i+2})B[k, i+1]+(t_{m+i+2}-t_j)B[k+1, i+1]$
    \State $B[k, i] \gets \dfrac{t_j-t_i}{t_{j+1}-t_i}B[k+1,i]+
                                          \dfrac{v_{mi}}{t_{j+1}-t_i}c_i$
  \EndFor
\EndFor
\State \Return $B$
\EndProcedure
\end{algorithmic}
\end{algorithm}

\section{Conclusions}                                    \label{S:Conclusions}

Two new differential-recurrence relations for the B-spline functions are given.
They can be used to relax the assumptions~\eqref{E:KnotClamped} about the knot
sequence and give an alternative proof to the differential-recurrence
relation~\eqref{E:BSplineDiffRecBase} for the B-spline functions of the same
degree which was given in~\cite{ChW2023}, along with an algorithm for computing
the coefficients of B-spline functions of degree $m$ in the adjusted
Bernstein-B\'{e}zier form. The algorithm~\cite[Thm 4.4, Alg.~4.1 and
\S6]{ChW2023} had a drawback --- the requirement to evaluate the coefficients
over multiple knot spans. We have eliminated this constraint, thus making the
computations independent for each knot span.

The new algorithm given in this manuscript is asymptotically optimal ---
computing $(m + 1)^2$ coefficients in $O(m^2)$ time --- and allows the user to
evaluate the coefficients only in the necessary knot spans. Moreover, if needed,
the computations for multiple knots spans can be done in parallel, as they are
fully independent. The method has good numerical behavior.

When the Bernstein-B\'{e}zier coefficients of B-spline basis functions are 
known, it is possible to compute any B-spline function in linear time with 
respect to its degree by performing the geometric algorithm proposed recently
in~\cite{WCh2020}. As shown in~\cite{ChW2023}, this can accelerate the 
evaluation when multiple B-spline curves or surfaces use the same knot sequence.

\bibliographystyle{elsart-num-sort}
\biboptions{sort&compress}
\bibliography{BSplineBF-one-rev-elsart}

\end{document}